\documentclass{article}   

\usepackage{amscd,amssymb,graphicx,stmaryrd,amsthm,color,amsmath, colonequals}   
\usepackage{hyperref} 
\usepackage{pgfplots}
\pgfplotsset{compat=newest}
\newlength\figurewidth
\newlength\figureheight
\definecolor{dg}{gray}{0.2}

\usepackage{nicefrac}

\DeclareMathOperator{\E}{\mathbb{E}}

\DeclareMathOperator{\Bin}{Bin}

\newtheorem{theorem}{Theorem}[section]  
\newtheorem{lemma}[theorem]{Lemma}

\theoremstyle{plain}

\begin{document} 

\title{Small Deviations of Sums of Independent Random Variables} 
\author{Brian Garnett} 
\date{}

\maketitle

\begin{abstract}
A well-known discovery of Feige's is the following \cite{feige}: Let $X_1, \ldots, X_n$ be nonnegative independent random variables, with $\E[X_i] \leq 1 \;\forall i$, and let $X = \sum_{i=1}^n X_i$.  Then for any $n$, 
\[\Pr[X < \E[X] + 1] \geq \alpha > 0,\]
for some $\alpha \geq 1/13$.  This bound was later improved to $1/8$ by He, Zhang, and Zhang \cite{4mom}.  By a finer consideration of the first four moments, we further improve the bound to approximately $.14$.  The conjectured true bound is $1/e \simeq .368$, so there is still (possibly) quite a gap left to fill.  
\end{abstract}


\section{Introduction} 
  
\subsection{A Small Deviation Inequalitiy}
Let $X_1, \ldots, X_n$ be nonnegative independent random variables, with $\E[X_i] = \mu_i \leq 1$ for each $i$.  For a given constant $\delta$, we wish to establish a universal lower bound 
\begin{equation}\label{mainprob}
\Pr \left[\sum_{i=1}^n X_i < \sum_{i=1}^n \mu_i + \delta \right] \geq \alpha > 0.
\end{equation}
Feige first established a bound of this type in \cite{feige}, and He, Zhang, and Zhang later showed that for $\delta \geq 1$, $\alpha \geq 1/8$ \cite{4mom}.  However, it is believed that in this case, we can let $\alpha = 1/e$.   If so, this bound would be tight, as consider letting all $X_i$ have mean $1$ and support $\{0,n+\delta\}$.  Then 
\[\displaystyle \Pr[X_1 + \ldots + X_n < n+\delta] = \left(1-\frac{1}{n+\delta}\right)^n \longrightarrow \frac{1}{e}.\]  
On the other hand, as Feige pointed out, for smaller $\delta$, the lower bound becomes dependent on this constant. Consider $X_1$ having mean $1$ and support $\{0,1+\delta\}$, and $X_i \equiv 1$ for $i \geq 2$.  In this case,
\begin{equation}\label{1delta}
\displaystyle \Pr[X_1 + \ldots + X_n < n+\delta] = \frac{\delta}{1+\delta}.
\end{equation}
    
It is not difficult to see that for such a small deviation from the mean, Markov's and Chebyshev's inequalities are insufficient for establishing a lower bound $\alpha$ in (\ref{mainprob}) away from $0$.  Therefore, we will need to consider more information than the just the first and second moments of our random variables.

\subsection{Our Results}
In Section 3, we establish the following bound:
\begin{theorem}\label{newbound}
Let $X_1, \ldots, X_n$ be nonnegative independent random variables, with $\E[X_i] \leq 1$ for each $i$.  Let $X = \sum_{i=1}^n X_i.$  Then 
\begin{equation}\label{14}
\Pr\left[X < \E[X] + 1\right] \geq \frac{7}{50}.
\end{equation}
\end{theorem}
In their approach to inequality (\ref{mainprob}), He, Zhang, and Zhang \cite{4mom} applied deviation inequalities they had developed in terms of the first, second, and fourth moments.  The source of our improvement comes from also considering the central third moment, and what happens in the cases where it is positive versus negative.  This idea is well illustrated by a (tight) moment bound we prove in Section \ref{moment}:
\begin{theorem}\label{simple}
Let $X$ be a random variable with $\E[X] = 0, \E[X^2] = \sigma^2$, and $\E[X^3] \geq 0$.  If $\E[X^4] \leq c\sigma^4$, then
\[ \Pr[X \geq 0] \leq 1 - \frac{1}{2c}. \]
\end{theorem}  
The assumption on the third moment allows for a slightly smaller bound than the one proved in \cite{bound}, which made no mention of the third moment (but otherwise had an identical hypothesis).  

We also consider whether we can obtain similar small deviation bounds if the random variables are only $k$-wise independent for some $k \geq 2$.  Recall that a collection of random variables is $k$-wise independent if any $k$-sized subcollection is mutually independent.  This is a natural consideration, since calculating up to the $k$th moment of a sum of independent random variables in fact only uses the assumption that they are $k$-wise independent.  In addition, for many randomized algorithms, $k$-wise independence is just as adequate as full independence, and the benefit of using the former is that it requires much less randomness to generate.  In this realm, we show that for certain types of random variables, $4$-wise independence is sufficient for a nontrivial small deviation bound.  Our most general result of this type, which we prove in Section \ref{Setup}, is
\begin{theorem}\label{third}
Let $X_1, \ldots, X_n$ be a $4$-wise independent collection of random variables where for each $i$, $\E[X_i] = 0$, and $|X_i| \leq 1$.  Let $X = \sum_{i=1}^n X_i$.  Then if $\delta \geq 1/3$,
\[ \Pr[X < \delta] \geq \frac{1}{6}. \]
\end{theorem}
In \cite{me}, we showed that $1/6$ is the best possible constant bound for this theorem.  Similar to the conjectured $1/e$ lower bound to (\ref{mainprob}) when $\delta \geq 1$, the bound $1/6$ cannot be improved by raising $\delta$ to a higher constant.  But in this case the bound does not hold when $\delta < 1/5$, due to the same example that produces (\ref{1delta}).  Thus, there may be some slight room for improvement to the above theorem, but not much.  But as we will see in our approach to Theorem \ref{newbound}, letting $\delta$ be as small as possible is a worthwhile endeavor. 

In Section \ref{23wise}, we present a counterexample to show that $3$-wise independence is insufficient for any nontrivial small deviation bound on a sum of random variables.  This settles a question in \cite{4mom}, regarding whether or not a nontrivial bound can be obtained from only the first, second, and third moments.  In addition, the assumption of pairwise independence does not lead to an improvement on Markov's inequality for a deviation bound on a sum of nonnegative random variables.

\begin{theorem}\label{2wisethm} 
Let $\delta > 0$.  If $(n+\delta)/(\delta+1) \in \mathbb{Z}$, then there exists a collection of nonnegative pairwise independent random variables $X_1, \ldots, X_n,$ each with mean 1 such that
\[\Pr[X_1 + \ldots + X_n < n+\delta] = \frac{\delta}{n+\delta}\]
\end{theorem}

\begin{theorem}\label{3wisethm}
Let $\delta > 0.$ If $(n+\delta)/(\delta+2) \in \mathbb{Z}$, then there exists a collection of nonnegative $3$-wise independent random variables $X_1, \ldots, X_n,$ each with mean 1 such that \[ \Pr[X_1 + \ldots + X_n < n+\delta] = \frac{(\delta+1)^2}{(\delta+2)(n+\delta)}.\]
\end{theorem}

\section{Setup}\label{Setup}

\subsection{A Moment Problem}\label{moment}
Let $X$ be a real-valued random variable.  Given information of the moments of $X$ up to some $k$, we want to bound the probability that $X$ lies in a set $S$.  This is a well-studied optimization problem that gives rise to an elegant dual problem, first utilized in \cite{isii} and \cite{karlin}, and treated extensively in \cite{bert}.  The setup of the general problem is
\begin{equation*}
\begin{aligned}
& \underset{X}{\text{maximize}}
& & \Pr[X \in S]  \\
& \text{subject to}
& & \E[X^i] = M_i, \; 0 \leq i \leq k.
\end{aligned}
\end{equation*}
Of course, $M_0 = 1$ always. The dual problem is then
\begin{equation*}
\begin{aligned}
& \underset{y}{\text{minimize}}
& & \sum_{i=0}^k q_i M_i  \\
& \text{subject to}
& & \sum_{i=0}^k q_i x^i \geq {\bf 1}_{\{x \in S\}}, \; \forall x \in \mathbb{R} 
\end{aligned}
\end{equation*}
In other words, this minimizes $\E[Q(X)]$ over all polynomials $Q$ of degree up to $k$, where $Q \geq {\bf 1}_S.$

For most of this paper, we let $k = 4$ and $S = \{x: x \geq \E[X] + \delta\}$.  Without loss of generality, assume $M_1 = \E[X] = 0$.  Thus for the dual problem, we need a polynomial $Q$ of degree at most $4$ such that $Q(x) \geq {\bf 1}_{\{x \geq \delta \}}(x)$ for all $x$.  The polynomial, which we will denote $Q_{\ell,r}$ for $\ell, r > 0$, we use throughout the paper will have the following properties:
\begin{figure}
  \centering
  \begin{tikzpicture}
    \begin{axis}[
        thick,
        color=dg,
        tick style={semithick, color=dg},
        tickwidth=0.3cm,
        width=\figurewidth,
        height=\figureheight,
        xmin=-1.6,
        xmax=1.6,
        ymin=-0.5,
        ymax=2.5,
        xtick={-1, 1},
        xticklabels={$-\ell$, $r$},
        ytick={1},
        y tick label style={yshift=2ex},
        legend style={draw=none},
        axis lines=middle
        ]
      \addplot[black, solid, very thick, <->]
      expression[domain=\pgfkeysvalueof{/pgfplots/xmin}+0.05:%
                        \pgfkeysvalueof{/pgfplots/xmax}-0.05,
                 samples=500]
      {0.5*x^4 - 0.25*x^3 - x^2 + 0.75*x + 1};
      \draw[dashed, thin, color=dg]
      (axis cs:\pgfkeysvalueof{/pgfplots/xmin},1) --
      (axis cs:\pgfkeysvalueof{/pgfplots/xmax},1);
    \end{axis}
  \end{tikzpicture}
  \caption{$Q_{\ell,r}$}
\end{figure}
\begin{itemize}
\item $Q_{\ell,r}(x)$ has a double root at $x = -\ell$.
\item $Q_{\ell,r}(0) = 1$ (we will often just shift by the small value $\delta$ when needed).
\item $Q_{\ell,r}(x) - 1$ has a double root at $x = r$.  
\end{itemize}
Most often, $r = \ell$, in which case we will denote it as $Q_r$.  In that case,
\begin{equation}\label{qr}
Q_r(x) = 1 + \frac{3}{4r} x - \frac{1}{r^2} x^2 - \frac{1}{4r^3} x^3 + \frac{1}{2r^4} x^4.
\end{equation}
We first use this approach to prove Theorem \ref{simple}, restated here:
\begin{theorem}
Let $X$ be a random variable with $\E[X] = 0, \E[X^2] = \sigma^2$, and $\E[X^3] \geq 0$.  If $\E[X^4] \leq c\sigma^4$, then
\[ \Pr[X \geq 0] \leq 1 - \frac{1}{2c}. \]
\end{theorem}
\begin{proof}
Consider the polynomial 
\[ Q(x) = Q_{\sqrt{c} \sigma}(x) = 1 + \frac{3}{4\sqrt{c}\sigma}x - \frac{1}{c \sigma^2}x^2 - \frac{1}{4c^{3/2}\sigma^3}x^3 + \frac{1}{2c^2\sigma^4}x^4, \]
which satisfies $Q(x) \geq 1_{\{x \geq 0\}}$ for all $x$ (we prove this for the more general expression of $Q_{\ell,r}$ in the next subsection).  Using the assumptions on the moments, we have
\[\Pr[X \geq 0] = \E[1_{X\geq 0}] \leq \E[Q(X)] \leq 1 - \frac{1}{c \sigma^2}\sigma^2 +  \frac{1}{2c^2\sigma^4}c\sigma^4 = 1 - \frac{1}{2c}.\]
\end{proof}
Note that the bound is tight if we consider, for any $a > 0$ and $p < 1/2$, 
\begin{equation}\label{ab}
 X = \begin{cases} -a, &\textrm{ with probability } p \\
0, &\textrm{ with probability } 1-2p \\
a, &\textrm{ with probability } p
\end{cases} 
\end{equation}
This happens to also be a tight example to Chebyshev's inequality.  Without any assumption on the third moment, He et al proved an upper bound of $1-(2\sqrt{3}-3)/c$ \cite{bound}.  Using our $Q_{\ell, r}$, we can choose $\ell = (1+\sqrt{3})r/2$ (which makes the degree-$3$ coefficient $0$) and optimize over $r$, to get the same bound.  

Many of our proofs will be of the same flavor as Theorem \ref{simple}, with $c = 3$ (which, not coincidentally, is the kurtosis of the normal distribution).  However, two complications will often arise, which one can predict by examining the idealistic conditions of the previous theorem.  Namely, the third moment could be negative, and the fourth moment may be a bit larger than $c\sigma^4$ for the optimal $c$ we are after.  Consider, for example, a sum of bounded independent random variables.

Let $\{X_i\}_{i \leq 1 \leq n}$ be independent random variables with $\E[X_i] = 0$ for each $i$. Let $X = \sum_{i=1}^n X_i.$  If $|X_i| \leq 1$ for each $i$, then
\begin{align*}
\left|\E[X^3]\right| &=  \left| \sum_{i=1}^n \E[X_i^3] \right| \\
&\leq \sum_{i=1}^n \E[|X_i|^3] \\
&\leq \sum_{i=1}^n \E[X_i^2] \\
&= \E[X^2].
\end{align*}
In addition,
\begin{align*}
\E[X^4] &= \sum_{i=1}^n \E[X_i^4] + 6 \sum_{i < j} \E[X_i^2]\E[X_j^2] \\
&= 3\left(\sum_{i=1}^n \E[X_i^2]\right)^2 + \sum_{i=1}^n \left(\E[X_i^4] - 3\E[X_i^2]^2 \right)\\
&= 3\E[X^2]^2 + \sum_{i=1}^n \left( \E[X_i^4] - 3\E[X_i^2]^2 \right) \\
&\leq 3\E[X^2]^2 + \sum_{i=1}^n \E[X_i^2] \\
&= 3\E[X^2]^2 +\E[X^2].
\end{align*}
So we see in this case that even if $\E[X^3]$ is negative, it can only be as low as $-\E[X^2]$, and $\E[X^4]$ can only exceed $3\E[X^2]^2$ by as much as $\E[X^2]$.  This will not present much of a problem asymptotically when the variance is large, but it will cause issues for small variances.  In that case, we just modify the polynomial.  However, in general, we cannot achieve the constant upper bound of $5/6$ unless we allow some deviation $\delta > 0$.  

Note that if $\{X_i\}_{i \leq 1 \leq n}$ are only $4$-wise independent, then $X$ will have the same moments above.  Since we will only use the first four moments of $X$ to prove Theorem \ref{third} and the related Lemmas in Section 3, we can assume the random variables are only $4$-wise independent.  In each situation, we will use this information on the moments to show there exist $\ell, r >0$ such that

\begin{equation}\label{Q56}
\E[Q_{\ell, r}(X-\delta)] \leq \frac{5}{6}. 
\end{equation}
Therefore, 
\begin{equation}\label{P56}
\Pr[X \geq \delta] = \E[{\bf 1}_{\{x \geq \delta\}}(X)] \leq \E[Q_{\ell, r}(X-\delta)] \leq \frac{5}{6},
\end{equation}
where the first inequality is shown in the next section.

\subsection{$Q_{\ell,r}$}\label{Qlr}
The polynomial $Q_{\ell,r}$ described in the previous section is explicitly given as
\begin{equation}\label{qlr}
Q_{\ell,r}(x) = \sum_{i=0}^n q_i x^i,
\end{equation}
where
\begin{align}\label{qis}
q_0 &= 1, \nonumber \\
q_1 &= \frac{2r^2(2\ell+r)}{\ell(\ell+r)^3}, \nonumber \\
q_2 &= \frac{r(-8\ell^2-\ell r + r^2)}{\ell^2(\ell+r)^3}, \nonumber \\
q_3 &= \frac{4\ell^2-4\ell r-2r^2}{\ell^2(\ell+r)^3}, \nonumber \\
q_4 &= \frac{3\ell+r}{\ell^2(\ell+r)^3}.
\end{align}
If $\ell = r$, then these coefficients simplify to
\begin{equation}\label{qisr}
q_0 = 1, \;
q_1 = \frac{3}{4r}, \;
q_2 = - \frac{1}{r^2} , \;
q_3 = - \frac{1}{4r^3}, \;
q_4 = \frac{1}{2r^4}.
\end{equation}

We will show directly that this polynomial satisfies 
\begin{lemma}
Let $\ell, r > 0$.  For all $x \in \mathbb{R}$, $Q_{\ell,r}(x) \geq 1_{\{x \geq 0\}}$.
\end{lemma}
\begin{proof}
\[Q_{\ell,r}(x) = \frac{1}{\ell^2 (\ell+r)^3} (\ell+x)^2 \big((\ell+r)^3 - 2(\ell^2+3\ell r + r^2) x + (3 \ell + r) x^2\big),\]
which is zero if $x=-\ell$. Otherwise, since 
\begin{align*}
\frac{\ell^2 (\ell+r)^3Q_{\ell,r}(x)}{(\ell+x)^2} &= (\ell+r)^3 - 2(\ell^2+3\ell r + r^2) x + (3 \ell + r) x^2 \\
&\geq (\ell+r)^3 - 2(\ell^2+3\ell r + r^2) \left(\frac{\ell^2+3\ell r + r^2}{3 \ell + r}\right) + (3 \ell + r)\left(\frac{\ell^2+3\ell r + r^2}{3 \ell + r}\right)^2 \\
&= \frac{(3 \ell + r)(\ell+r)^3 - (\ell^2+3\ell r + r^2)^2}{3 \ell + r} \\
&= \frac{2\ell^4+4\ell^3 r + \ell^2 r^2}{3 \ell + r} \geq 0,
\end{align*}
we have $Q_{\ell,r}(x) \geq 0$ $\forall x$.  On the other hand,
\begin{align*}
Q_{\ell,r}(x) - 1 &=  \frac{1}{\ell^2 (\ell+r)^3} x \big(4\ell^2+2\ell r + (3\ell + r)x \big)(-r+x)^2 \\
&\geq 0,  \textrm{    for all } x \geq 0.
\end{align*}
\end{proof}
Now, let $\E[X] = 0,$ and $\delta > 0$.  Then
\begin{align}\label{xdelta}
\E[Q_{\ell,r}(X-\delta)] &= \sum_{i=0}^4 q_i \E[(X-\delta)^i] \nonumber \\
&= \sum_{i=0}^4 q_i \E[(X)^i] + \sum_{i=1}^4 (-1)^i \delta^i q_i + (6\delta^2 q_4 - 3\delta q_3)\E[X^2] - 4\delta q_4 \E[X^3] \nonumber \\
&= \sum_{i=0}^4 (-1)^i \delta^i q_i + (q_2 - 3\delta q_3  + 6\delta^2 q_4)\E[X^2] + (q_3 - 4\delta q_4)\E[X^3] + q_4 \E[X^4].
\end{align}
Looking at the coefficients in (\ref{qis}), notice that $q_4 > 0$ always, and if $\ell \leq (1+\sqrt{3})r/2$, then $q_3 < 0$.  In fact, we will always choose $\ell \leq r$.  Therefore, we will always have
\begin{equation}\label{q3q4}
q_3 < 0, \text{ and } q_4 > 0.
\end{equation}
Thus, if $X$ satisfies the inequalities (\ref{thirdfourth}) below, then
\begin{equation}\label{xdelta1}
\E[Q_{\ell,r}(X-\delta)] \leq \sum_{i=0}^4 (-\delta)^i q_i + (q_2 - 3\delta q_3  + 6\delta^2 q_4)\sigma^2 + (q_3 - 4\delta q_4)(-\sigma^2) + q_4 (3\sigma^4 + \sigma^2).
\end{equation}
We will often let $\ell = r = \sqrt{3}\sigma$. In that case, (\ref{qisr}) becomes
\begin{equation}\label{rt3sigma}
q_0 = 1, \;
q_1 = \frac{3}{4\sqrt{3}\sigma}, \;
q_2 = - \frac{1}{3\sigma^2} , \;
q_3 = - \frac{1}{12\sqrt{3}\sigma^3}, \;
q_4 = \frac{1}{18\sigma^4}.
\end{equation}
Substituting these into (\ref{xdelta1}) and simplifying yields 
\begin{equation}\label{sqrt3sig}
\E[Q_{\sqrt{3}\sigma}(X-\delta)] \leq \frac{5}{6} + \frac{2\delta^4+\sqrt{3}\delta^3\sigma+(2+8\delta)\sigma^2+(\sqrt{3}-6\sqrt{3}\delta)\sigma^3}{36\sigma^4}.
\end{equation}

\subsection{Proof of Theorem \ref{third}}
We restate it here in an equivalent form.
\begin{theorem}
Let $X_1, \ldots, X_n$ be a $4$-wise independent collection of random variables where for each $i$, $\E[X_i] = 0$, and $|X_i| \leq 1$.  Let $X = \sum_{i=1}^n X_i$.  Then
\[ \Pr[X \geq 1/3] \leq \frac{5}{6}. \]
\end{theorem}
\begin{proof}
Let $X_1, \ldots, X_n$ be $4$-wise independent random variables with $\E[X_i] = 0$ and $|X_i| \leq 1$ for each $i$.  Let $X = \sum_{i=1}^n X_i$ (so that $\E[X] = 0$), and let $\sigma^2 = \E[X^2]$. At the end of Section \ref{moment}, we showed that since $|X_i| \leq 1$ for each $i$,
\begin{align}\label{thirdfourth}
\E[X^3] &\geq -\sigma^2, \nonumber \\
\E[X^4] &\leq 3\sigma^4 + \sigma^2.
\end{align} 
As explained in the same section, it is sufficient to show that for any such $X$, there is a choice of $\ell$ and $r$ such that
\[\E[Q_{\ell, r}(X - \delta)] \leq \frac{5}{6}.\]
For this proof, we can let $\ell = r$ for each case, so we refer to the polynomial as $Q_r$.  First, let $\ell = r = \sqrt{3}\sigma$.  Using (\ref{sqrt3sig}) with $\delta = 1/3$, we have
\[\E[Q_{\sqrt{3}\sigma}(X-1/3)] \leq \frac{5}{6} + \frac{1}{36\sigma^4}\left(\frac{2}{81}+\frac{\sqrt{3}}{27}\sigma+\frac{14}{3}\sigma^2-\sqrt{3}\sigma^3\right).\]
If $\sigma \geq 3$, then
\begin{align*}
\E[Q_{\sqrt{3}\sigma}(X-1/3)] &\leq \frac{5}{6} + \frac{1}{36\sigma}\left(\frac{2}{81}\sigma^{-3}+\frac{\sqrt{3}}{27}\sigma^{-2}+\frac{14}{3}\sigma^{-1}-\sqrt{3}\right)\\
&\leq \frac{5}{6} + \frac{1}{36\sigma}\left(\frac{2}{81}3^{-3}+\frac{\sqrt{3}}{27}3^{-2}+\frac{14}{3}3^{-1}-\sqrt{3}\right) \\
&\leq \frac{5}{6}.
\end{align*}
Now If we let $r = a\sigma$ for a constant $a > 0$, and $\delta = 1/3$, putting the coefficients of $Q$ (\ref{qisr}) into (\ref{xdelta1}) yields
\begin{equation*}
\E[Q_{a\sigma}(X-1/3)] \leq \frac{3-2a^2+2a^4}{2a^4} + \frac{2-a^2}{4a^3}\sigma^{-1} + \frac{27-2a^2}{18a^4} \sigma^{-2} + \frac{1}{108a^3} \sigma^{-3} + \frac{1}{162a^4} \sigma^{-4}.
\end{equation*}
Let $B_a(\sigma)$ be the quantity on the righthand side.  Examining the coefficients, we see that if $27 - 2a^2 \geq 0$, then $B_a$ is a convex polynomial in the variable $\sigma^{-1}$.  Thus, for a fixed $a$, and $\sigma_1 < \sigma_2$, if we show that $B_a(\sigma_1)$ and $B_a(\sigma_2)$ are both bounded above by $5/6$, then $\E[Q_{a\sigma}(X-1/3)] \leq 5/6$ for all $\sigma \in [\sigma_1, \sigma_2]$.    

First, let $a = 2$.  Then
\[ B_2(\sigma) = \frac{27}{32} - \frac{1}{16} \sigma^{-1} + \frac{19}{288} \sigma^{-2} + \frac{1}{864} \sigma^{-3} + \frac{1}{2592} \sigma^{-4}. \]
and it can be easily checked that $B_2(3/2) < 5/6$ and $B_2(3) < 5/6$. \\
If $a = 9/4$, then
\[ B_{9/4}(\sigma) = \frac{1883}{2187} - \frac{49}{729} \sigma^{-1} + \frac{80}{2187} \sigma^{-2} + \frac{16}{19683} \sigma^{-3} + \frac{128}{531441} \sigma^{-4}, \]
with $B_{9/4}(1) < 5/6$ and $B_{9/4}(3/2) < 5/6$.\\
If $a = 5/2$, then
\[ B_{5/2}(\sigma) = \frac{549}{625} - \frac{17}{250} \sigma^{-1} + \frac{116}{5625} \sigma^{-2} + \frac{2}{3375} \sigma^{-3} + \frac{8}{50625} \sigma^{-4}, \]
with $B_{5/2}(1/2) < 5/6$ and $B_{5/2}(1) < 5/6$.

Thus, we have covered all $\sigma \geq 1/2.$  Lastly, we set $r=3/2$, for which substituting (\ref{qisr}) into (\ref{xdelta1}) gives
\[\E[Q_{3/2}(X-1/3)] \leq \frac{10339}{13122} + \frac{8}{27}\sigma^4 \leq \frac{5}{6},\]
when $\sigma < 1/2$.  
\end{proof}
As we discussed in the introduction, the deviation $\delta = 1/3$ could possibly be lowered, but not to anything below $1/5$.  However, due to the small variance case, our approach cannot allow for a $\delta$ much lower than the one we set.

\section{Proof of Main Theorem}
In this Section, we prove Theorem \ref{newbound}.  First, we will need two modified versions of Theorem $\ref{third}$.  Although we will have full independence when we apply these lemmas, we only assume $4$-wise independence for maximal generality.  We treat separately the cases of negative and nonnegative third moment.  The (rather tedious) proofs of both lemmas are at the end of the section.
\subsection{Lemmas}
Due to the third-degree coefficient $q_3$ of our polynomial being negative, if we know the central third moment is positive, we can lower the allowed deviation $\delta$ from $1/3$, while keeping the same upper bound of $5/6$ on the probability.  It will be important to lower $\delta$ as much as possible, without having to raise the bound on the probability (which would not be a good tradeoff).
\begin{lemma}\label{425}
Let $X_1, \ldots, X_n$ be a $4$-wise independent collection of random variables where for each $i$, $\E[X_i] = 0$, and $|X_i| \leq 1$.  Let $X = \sum_{i=1}^n X_i$.  If $\E[X^3] \geq 0$, then
\[ \Pr[X \geq 4/25] \leq \frac{5}{6}. \]
\end{lemma}

For the next lemma, we will assume each random variable is supported on two points; this will be the case when we apply it in the upcoming proof.  Now, if we assume the central third moment of the sum is nonpositive and add one small condition, we can remove the assumption of a universal upper bound (intuitively, a negative central third moment implies the distributions of the random variables are already skewed below their means).  This will also be a crucial component to the proof of the theorem.  

\begin{lemma}\label{negthird}
Let $X_1, \ldots, X_n$ be a $4$-wise independent collection of random variables where for each $i$, $\E[X_i] = 0$, and $X_i$ has support $\{-a_i, b_i\}$.  Assume that $a_i \leq 1$ for each $i$, $b_1 = \max_{i}\{b_i\},$ and $a_1 \geq 1/16$.  Let $X = \sum_{i=1}^n X_i$.  If $\E[X^3] \leq 0$, then
\[ \Pr[X \geq 1] \leq \frac{5}{6}. \]
\end{lemma}
The allowed deviation of $1$ and the $1/16$ assumption above can be tinkered with, but we fixed $\delta = 1$ in preparation for the theorem.

\subsection{Proof of Theorem \ref{newbound}}
We state it again, this time with a slightly better but also less nice-looking constant:
\begin{theorem}
Let $X_1, \ldots, X_n$ be nonnegative independent random variables with means $\mu_1, \ldots, \mu_n$ such that $\mu_i \leq 1$ for every $i$.  Then 
\begin{equation}\label{14}
\Pr\left[\sum_{i=1}^n X_i < \sum_{i=1}^n \mu_i + 1\right] \geq \beta,
\end{equation}
where we set $\beta = \displaystyle \frac{46}{279}e^{-4/25} \; \left(>\frac{7}{50}\right)$.
\end{theorem}
In his proof \cite{feige} which first established a lower bound on this probability, Feige explained via a linear programming argument that without loss of generality, we may assume that each $X_i$ is non-constant and has support of size two.  This was one aspect of his overall strategy, which was to apply a sequence of transformations to the collection of random variables, where each transformation does not increase the probability that we wish to lower bound.  The next step is to simply subtract some nonnegative amount from each $X_i$, so that it has support $\{0, c_i\}$ for some $c_i > 0$.  This step may reduce the mean $\mu_i$ but leaves the probability in (\ref{14}) unchanged.

The goal of the next transformation, which he called ``merge," was to make the means closer to one another.  With ``merge," we take the two random variables with the smallest means, say $X_i$ and $X_j$ with means $\mu_i$ and $\mu_j$, and merge them into the random variable $X' = X_i + X_j$ with mean $\mu' = \mu_i + \mu_j$.  Now $X'$ possibly has support of size up to $4$, but as before, we may reduce its size to two and align it with $0$.  For some threshold $t \leq 1/2$, we will apply ``merge" (followed by reducing the support and aligning with $0$) on the two random variables with smallest means, $\mu_i < \mu_j$, if and only if $\mu_i < t$ and $\mu_j \leq 1-t$.  Thus, we will never create a random variable with a mean larger than $1$.  Furthermore, when we have finished these transformations, we have at most one random variable with mean below $t$, in which case all other means are above $1-t$. 

\begin{proof} 
As explained in the precursor to this proof, we may assume that each $X_i$ has support $\{0, c_i\}$ for some $c_i > 0$, so that $\Pr[X_i = c_i] = \mu_i/c_i.$ For each $i$, let $s_i = c_i - \mu_i$, the ``surplus" to the mean.  We may assume
\[s_1 \geq \ldots \geq s_n.\] 
Using a trick from \cite{4mom}, fix $\tau > 0$, and define
\[ k = \max\left\{0, \max_{1 \leq j \leq n} \{j : s_j \geq \tau (\mu_1 + \ldots \mu_j) \}\right\}. \]
Let $m = \sum_{i=1}^k \mu_i$, the mean of the sum of the first $k$.  If $i > k$, then
\begin{equation}\label{bbound}
 s_i \leq s_{k+1} \leq \tau \sum_{i=1}^{k+1} \mu_i \leq \tau(m+\mu_{k+1}) \leq \tau(m+1). 
\end{equation}
Otherwise, if $i \leq k$, $s_i \geq s_k \geq \tau m.$  If $k > 0$, then
\begin{align*}
\Pr\left[\sum_{i=1}^k X_i = 0\right] &= \prod_{i=1}^k \Pr[X_i = 0] \\
&= \prod_{i=1}^k \left(1-\frac{\mu_i}{c_i}\right) \\
&= \prod_{i=1}^k \left(1-\frac{\mu_i}{s_i + \mu_i}\right) \\
&\geq \prod_{i=1}^k \left(1-\frac{\mu_i}{\tau m + \mu_i}\right) \\
&\geq \prod_{i=1}^k e^{-\mu_i/(\tau m)} = e^{-1/\tau}.
\end{align*}
The utility of this splitting of the random variables is that conditioning on the sum of first $k$ being $0$, the rest are bounded by an amount comparable to the allowed deviation.  Here in particular, we are using full (as opposed to just $4$-wise) independence of the random variables (we also implicitly used full independence during the merge operation described above).
\begin{align*}
\Pr\left[\sum_{i=1}^n X_i < \sum_{i=1}^n \mu_i + 1\right] &\geq \Pr\left[\sum_{i=1}^k X_i = 0\right] \cdot \Pr\left[\sum_{i=k+1}^n X_i < \sum_{i=1}^n \mu_i + 1\right] \\
&\geq e^{-1/\tau} \Pr\left[\sum_{i=k+1}^n X_i < \sum_{i=1}^n \mu_i + 1\right] \\
&= e^{-1/\tau} \Pr\left[\sum_{i=k+1}^n X_i < \sum_{i=k+1}^n \mu_i + (m+1)\right],
\end{align*}
and we will now focus on the latter probability.  We fix $\tau = 25/4$.  Assume $k < n$, otherwise we are done.  For $1 \leq j \leq n-k$, let $Y_j = X_{k+j} - \mu_{k+j}$, and let $n' = n - k$.  Set $Y = \sum_{j=1}^{n'} Y_j$.  Each $Y_j$ has mean $0$ and support $\{-a_j, b_j\}$ where $0 < a_j \leq 1$ and $0 < b_j \leq 25(m+1)/4.$  We break the analysis into two cases, depending on the sign of the third moment of $Y$.

\emph{Case 1}: $\E[Y^3] \geq 0.$ \\
In this case, for each $j$, let $Y'_j = \displaystyle \frac{4}{25(m+1)} Y_j,$ and $Y' = \sum_{i=1}^{n'} Y'_j.$  Note that $\E[(Y')^3] \geq 0$, and for each $j$, $|Y'_j| \leq 1$. By Lemma \ref{425},
\begin{equation*}
\Pr\left[\sum_{j=1}^{n'} Y_j < (m+1)\right] = \Pr\left[\sum_{j=1}^{n'} Y'_j < 4/25 \right] \geq \frac{1}{6}.
\end{equation*}
Thus, we have
\begin{equation*}
\Pr\left[\sum_{i=1}^n X_i < \sum_{i=1}^n \mu_i + 1\right] \geq \frac{e^{-4/25}}{6} > \beta.
\end{equation*}
Although proving this case was immediate, it required the bounding of the latter random variables and drove the choice of $\tau = 25/4$.  

\emph{Case 2}:  $\E[Y^3] < 0.$ \\
The major fact about Lemma \ref{negthird} we use in this case is that we do not need an upper bound on the $Y_j$'s.  Above we had to divide the random variables by some amount in order to apply our positive third moment lemma, which lowered the allowed deviation in our strict application of the statement.  This time, we do not have to do so, and the allowed deviation $\delta$ remains at least $1$.      

Now, each $Y_j$ has support $\{-a_j, b_j\}$, where $0 < a_j \leq 1$ for each $i$. Since $b_j = s_{k+j}$, we have $b_1 \geq \ldots \geq b_{n'}$.  If $a_1 \geq 1/16$, we can immediately apply Lemma \ref{negthird}, and we are done.  So we can assume $a_1 < 1/16$.  Since $Y_j = X_{k+j} - \mu_{k+j}$, each $a_j = \mu_{k+j}$.  Thus, $\mu_{k+1} < 1/16$.  By the stopping condition of the merge process, this means that all other means exceed $15/16$.   

We may also assume at this point that $b_1 \geq 3(m+1)$.  Otherwise, like in Case 1, we can divide by $3(m+1)$, and by Theorem \ref{third},
\begin{equation*}
\Pr\left[\sum_{j=1}^{n'} Y_j < (m+1)\right] = \Pr\left[\sum_{j=1}^{n'} Y'_j < 1/3 \right] \geq \frac{1}{6}.
\end{equation*}
Thus, considering
\[ \Pr[Y_1 = b_1] = \frac{a_1}{a_1+b_1} \leq \frac{1}{48(m+1)},\]
this variable being positive is quite unlikely, and in order to discard it, we will also condition on this not occurring.  Once we do so, we must take note that the third moment of the remaining sum is also negative, as we have subtracted from it
\[E[Y_1^3] = a_1 b_1 (b_1 - a_1) > a_1 b_1 (3 - 1/16) > 0.\]
Furthermore, for $j \geq 2,$ $a_j > 1/16$ (in fact $a_j \geq 15/16$), so we can apply Lemma \ref{negthird} to the remaining sum.  Now we consider two cases: $k = 0$ and $k \geq 1$.

If $k=0$, then $m=0$, and
\begin{align*}
\Pr \left[\sum_{i=1}^n X_i < \sum_{i=1}^n \mu_i + 1 \right] &= \Pr \left[\sum_{j=1}^{n} Y_j < 1 \right] \\
&\geq \Pr[Y_1 = 0] \cdot \Pr \left[\sum_{j=2}^{n} Y_j < 1 \right] \\
&\geq \left(1-\frac{1}{48}\right) \cdot \frac{1}{6} \quad \text{ (by Lemma \ref{negthird})} \\
& > \beta
\end{align*}
If $k \geq 1$, $m \geq \E[X_1] \geq 15/16$, and
\begin{align*}
\Pr\left[\sum_{i=1}^n X_i < \sum_{i=1}^n \mu_i + 1\right] &\geq e^{-4/25} \Pr\left[\sum_{j=1}^{n'} Y_j < m+1 \right] \\
&\geq e^{-4/25} \Pr[Y_1 = 0]\cdot \Pr\left[\sum_{j=2}^{n'} Y_j < m+1 \right] \\
&\geq e^{-4/25} \left(1-\frac{1}{48(m+1)}\right) \cdot \Pr\left[\sum_{j=2}^{n'} Y_j < 1 \right] \\
&\geq e^{-4/25} \left(1-\frac{1}{48(m+1)}\right) \cdot \frac{1}{6} \quad \text{(by Lemma \ref{negthird})}\\
&\geq e^{-4/25} \left(\frac{92}{93} \right) \left(\frac{1}{6}\right) = \beta.
\end{align*}
\end{proof}
We remark that given the tightness of Theorem \ref{third} and the lemmas in this section (which we showed in \cite{me}), one cannot achieve a constant higher than $1/6$ in Theorem \ref{newbound} with only the information of the first four moments.  The room for improvement in our work lies in the possibility of lowering $\delta = 4/25$ in Lemma \ref{425}.  We could not do so (by more than a negligible amount) in our proof below.  However, perhaps a deeper analysis could allow it.  

Furthermore, we believe a tractable approach to bridging some of the gap between our $7/50$ and the conjectured $1/e$ would be to apply a similar $2k$th moment method.  An effective dual $2k$-degree polynomial $Q$ may be similarly defined as our $Q_{\ell,r}$ but possibly with more double roots for $Q$ and $Q-1$.  In addition, $Q$ could be defined so that many of its odd-degree coefficients are $0$, at the benefit of disregarding the odd moments of those orders.

\subsection{Proofs of Lemmas}
As explained at the end of Section \ref{moment}, we will show that for any $X$ meeting the conditions, there is a choice of $\ell$ and $r$ such that (\ref{Q56}) and thus (\ref{P56}) hold.  We will also refer to properties of the polynomial $Q_{\ell, r}$ laid out in Section \ref{Qlr}.
\subsubsection{Proof of Lemma \ref{425}}
\begin{proof}
Let $X_1, \ldots, X_n$ be $4$-wise independent random variables such that for each $i$, $\E[X_i] = 0$, and $|X_i| \leq 1.$ Let $X = \sum_{i=1}^n X_i$.  This time, by assumption, we have $\E[X^3] \geq 0$.  Otherwise, $\E[X] = 0$ and $\E[X^4] \leq 3\sigma^4+\sigma^2$, as shown in Section \ref{moment}.  Now, from (\ref{xdelta}) and (\ref{q3q4}) we have
\begin{equation}\label{xdelta2}
\E[Q_{\ell,r}(X-\delta)] \leq \sum_{i=0}^4 (-\delta)^i q_i + (q_2 - 3\delta q_3  + 6\delta^2 q_4)\sigma^2  + q_4 (3\sigma^4 + \sigma^2).
\end{equation}
With $\ell = r = \sqrt{3}\sigma$,
\begin{equation*}
\E[Q_{\ell,r}(X-\delta)] \leq \frac{5}{6} + \frac{2\delta^4+\sqrt{3}\delta^3\sigma + 2\sigma^2-6\sqrt{3}\delta\sigma^3}{36\sigma^4}.
\end{equation*}
Letting $\delta = 4/25$ and $\sigma > 5/4$,
\begin{align*}
\E[Q_{\sqrt{3}\sigma}(X-4/25)] &\leq \frac{5}{6} + \frac{1}{36\sigma}\left(2\left(\frac{4}{25}\right)^4\sigma^{-3}+\sqrt{3}\left(\frac{4}{25}\right)^3\sigma^{-2}+2\sigma^{-1}-\frac{24\sqrt{3}}{25}\right)\\
&\leq  \frac{5}{6} + \frac{1}{36\sigma}\left(2\left(\frac{4}{25}\right)^4\left(\frac{4}{5}\right)^3+\sqrt{3}\left(\frac{4}{25}\right)^3\left(\frac{4}{5}\right)^2+2\left(\frac{4}{5}\right)-\frac{24\sqrt{3}}{25}\right)\\
&\leq \frac{5}{6}.
\end{align*}
For $\sigma \in [0,5/4]$, we will be forced to choose $\ell < r$. In order to mitigate some of the upcoming messiness, we refer to $\delta = 4/25$ as $\delta$.

Let $\ell = 2\sigma$ and $r = 5\sigma/2$.  Then using (\ref{xdelta2}) and (\ref{qis}), 
\[ \E[Q_{2\sigma, 5\sigma/2}(X-\delta)] \leq \frac{835}{972} - \frac{226\delta}{729}\sigma^{-1} + 
\frac{68-207\delta^2}{2916} \sigma^{-2} + \frac{11\delta^3}{243}\sigma^{-3} + \frac{17\delta^4}{729}\sigma^{-4}. \]
Since $68-207\delta^2 \geq 0$, the right-hand side is a convex polynomial of the variable $\sigma^{-1} > 0$.  One can check that when $\sigma = .68$ and when $\sigma = 1.25$ (and $\delta = .16$), it is less than $5/6$.  Therefore, 
\[\E[Q_{2\sigma, 5\sigma/2}(X-4/25)] \leq \frac{5}{6} \]
for all $\sigma \in [.68, 1.25].$ \\
Next, let $\ell = 15\sigma/7$ and $r = 3\sigma$.  Then
\[ \E[Q_{15\sigma/7, 3\sigma}(X-\delta)] \leq \frac{256957}{291600} - \frac{10633\delta}{32400}\sigma^{-1} + \frac{26411-128625\delta^2}{1749600} \sigma^{-2} + \frac{7889\delta^3}{194400}\sigma^{-3} + \frac{26411\delta^4}{1749600}\sigma^{-4}. \]
Again, this is a convex polynomial of the variable $\sigma^{-1} > 0$, since the coefficient of $\sigma^{-2}$ is positive for $\delta = 4/25$.  One can check that when $\sigma = .5$ and when $\sigma = .68$, the right-hand side is less than $5/6$.  Therefore, 
\[ \E[Q_{(15\sigma/7), 3\sigma}(X-4/25)] \leq \frac{5}{6} \]
for all $\sigma \in [.5, .68].$ \\
Lastly, let $\ell = 1$ and $r=2$.  Then 
\[ \E[Q_{1,2}(X-4/25) ] \leq \frac{62573}{78125} - \frac{59}{3375} \sigma^2 + \frac{5}{9} \sigma^4. \]
One can verify with the quadratic formula or by other means that the right-hand side is bounded above by $5/6$ when $\sigma \in [0,1/2].$ 

Overall, we have provided a suitable polynomial $Q$ for every $\sigma \geq 0$.  
\end{proof}

\subsubsection{Proof of Lemma \ref{negthird}}
\begin{proof}
Let $X_1, \ldots, X_n$ be $4$-wise independent mean-zero random variables distributed as 
\[ X_i = \begin{cases}    -a_i, &\textrm{ with probability } \displaystyle  \frac{b_i}{a_i+b_i} \\
b_i, &\textrm{ with probability } \displaystyle \frac{a_i}{a_i+b_i},
\end{cases} \]
where $a_i \leq 1$ for each $i$, $b_1 = \max_{i}\{b_i\}$, and $a_1 \geq 1/16$.  Then

\[ \E[X] = 0, \]
\[E[X^2] \colonequals \sigma^2 = \sum_{i=1}^n a_i b_i \,, \]
\begin{align*}
\E[X^3] &= \sum_{i=1}^n a_i b_i(b_i - a_i) \\
&\geq -\sum_{i=1}^n a_i^2 b_i \\
&\geq -\sum_{i=1}^n a_i b_i = -\sigma^2, 
\end{align*}
\begin{align*}
\E[X^4] &\leq 3\sigma^4 + \sum_{i=1}^n \E[X_i^4] \\
&= 3\sigma^4 + \sum_{i=1}^n a_i b_i(a_i^2+b_i^2) \\
&\leq 3\sigma^4 + \sum_{i=1}^n a_i^3 b_i + \sum_{i=1}^n a_i b_i^3 \\
&\leq 3\sigma^4 + \sigma^2 + \sum_{i=1}^n a_i b_i^3
\end{align*}

Now we will show $\sum_{i=1}^n a_i b_i^3 \leq 4\sigma^3.$  Despite the lack of an upper bound on the $b_i$'s, the nonpositivity of the third moment, along with the prescribed interval of $a_1$, brings the sum under control (the latter condition, simply put, prevents an extremely large $b_1$ being ``hidden" by an extremely small $a_1$).  First, note that $\E[X^3] \leq 0$ implies
\begin{equation*}
\sum_{i=1}^n a_i b_i^2 \leq \sum_{i=1}^n a_i^2 b_i \leq \sigma^2.
\end{equation*}
Then
\begin{align*}
\sum_{i=1}^n a_i b_i^3 &\leq b_1 \sum_{i=1}^n a_i b_i^2 \\
&\leq b_1 \sigma^2 \\
&= \sqrt{b_1^2} \sigma^2 \\
&= \left(\frac{1}{\sqrt{a_1}} \sqrt{a_1 b_1^2}\right) \sigma^2 \\
&\leq \left(4 \sqrt{\sum a_i b_i^2}\right) \sigma^2 \\
&\leq (4 \sqrt{\sigma^2}) \sigma^2 = 4 \sigma^3.
\end{align*}
From (\ref{xdelta}) and (\ref{q3q4}), we have
\begin{equation}\label{xdelta3}
\E[Q_{\ell,r}(X-1)] \leq \sum_{i=0}^4 (-1)^i q_i + (q_2 - 3q_3  + 6 q_4)\sigma^2 + (q_3 - 4q_4)(-\sigma^2) + q_4 (3\sigma^4 + 4\sigma^3+ \sigma^2).
\end{equation}
If $\ell = r = \sqrt{3}\sigma$, 
\begin{equation*}
\E[Q_{\sqrt{3}\sigma}(X-1)] \leq \frac{5}{6} + \frac{2+\sqrt{3}\sigma + 10\sigma^2+(8-5\sqrt{3})\sigma^3}{36\sigma^4}.
\end{equation*}
If $\sigma \geq 16$, 
\begin{equation*}
\E[Q_{\sqrt{3} \sigma}] \leq \frac{5}{6} + \frac{1}{36\sigma}(2(16)^{-3}+\sqrt{3}(16)^{-2} + 10(16)^{-1} + 8 - 5\sqrt{3}) \leq \frac{5}{6}.
\end{equation*}
For the rest of this proof, we will still have $\ell = r.$  We will find a $Q_r$ for each $\sigma \in [0,16].$  First, let $r = 19\sigma/10$.  From (\ref{xdelta3}), we have
\[\E[Q_{(19\sigma/10)}(X-1)] \leq \frac{218442-24885\sigma^{-1}+37800\sigma^{-2}+9500\sigma^{-3}+10000\sigma^{-4}}{260642}.\]
As in the cases in the other proofs, the right-hand side is a convex polynomial of the parameter $\sigma^{-1} > 0$.  One can check that for $\sigma = 5/2$ and $\sigma = 16$, the right-hand side is less than $5/6$.  Therefore, 
\[\E[Q_{(19\sigma/10)}(X-1)] \leq \frac{5}{6} \]
when $\sigma \in [5/2, 16].$\\
For the remaining $\sigma$, we can let $r=5$.  Then (\ref{xdelta3}) becomes
\begin{align*}
\E[Q_5(X-1)] &\leq \frac{1016 - 29\sigma^2 + 4\sigma^3 + 3\sigma^4}{1250}\\
& \leq \frac{5}{6} \textrm{ when } \sigma \in [0,5/2],
\end{align*}
and the last inequality can be verified using basic calculus.  All cases are covered.  
\end{proof}

\section{$2$- and $3$-wise independent counterexamples}\label{23wise}
\subsection{Setup}
Using notation from \cite{all1}, let $\mathcal{A}(n,k,p)$ be the set of all collections of $n$ $k$-wise independent Bernoulli random variables with equal marginal probabilities $p$.  We denote
\begin{equation}\label{newprimal}
Z_P(n,k,p,\delta) = \max_{(X_1, \ldots, X_n)  \in \mathcal{A}(n,p,k)} \Pr[X_1 + \ldots X_n \geq np + \delta].
\end{equation}
We can find the above quantity using linear programming. Let $S = X_1 + \ldots + X_n$.  Since we are interested in the symmetric event $\{S \geq np + \delta\}$, there is no loss in assuming that our identically distributed random variables are also symmetric.  Hence, our programming problem will be in the $n+1$ variables $p_0, \ldots, p_n$, where
\begin{equation}\label{pr}  
p_r \colonequals \Pr[S = r]. 
\end{equation}
Let $X \sim \Bin(n,p)$.  For $k$-wise independence to hold, it is sufficient for the moments of $S$ and $X$ to be identical up to order $k$.  Thus, we have the constraints
\begin{equation}\label{constr}
\E[X^i] = \E[S^i] = \sum_{r=0}^n r^i p_r, 
\end{equation}
for $0 \leq i \leq k$.  Letting $m=\lceil np + \delta \rceil$, our objective function is $\sum_{r = m}^n p_r.$ The dual problem is then
\begin{equation}\label{newdual}
Z_D(n,k,p,\delta) = \min_{Q \in \mathcal{P}_k} \E_{X \sim Bin(n,p)}[Q(X)],
\end{equation} 
where $\mathcal{P}_k$ is the set of univariate polynomials $Q$ of degree at most k, with 
\begin{align}
Q(i) &\geq 0 \; \forall i \in \{0, \ldots, m(d)-1\}, \textrm{ and } \\
Q(j) &\geq 1 \; \forall j \in \{m(d), \ldots, n\}.
\end{align}
By linear programming duality, $Z_P = Z_D \; (\colonequals Z)$. As explained in \cite{all1}, an optimal $Q_0$ in (\ref{newdual}) would give us information about the optimal distribution $S$ in the primal problem (\ref{newprimal}).  Assuming optimality of each,
\begin{equation*}
\sum_{i=m}^n \Pr[S=i] = Z = \E[Q_0(S)] = \sum_{i=0}^n Q_0(i) \Pr[S=i].
\end{equation*}
Thus, the support of $S$ contains only integers which are zeros of $Q_0$ as well as the $i \geq m$ where $Q_0(i) = 1$.  With this information, one can simply use the $k+1$ linear constraints to solve for the probabilities.

\subsection{$2$-wise}
We will find it convenient to set $\delta = dp$, so that $m = \lceil np + dp \rceil = \lceil (n+d)p \rceil $.
For $k = 2$, the optimal solution occurs when
\begin{align*}
p_0 &= \frac{(1-p)(m-np+p)}{m}, \\
p_m &= \frac{p(1-p)n(n-1)}{m(n-m)}, \\
p_n &= \frac{p(np-m+1-p)}{n-m},
\end{align*}
valid as long as $m \leq np + 1-p$.  Then
\begin{equation}\label{max2}
\Pr[S \geq m] = \frac{p(n+m-np-1+p)}{m}.
\end{equation}
The optimal polynomial in the corresponding dual problem (\ref{newdual}) is
\begin{equation*}
f(x) = \frac{1}{mn}\big((m+n)x - x^2\big).
\end{equation*}
Note that $f$ satisfies the conditions, $f(0) = 0, f(m) = f(n) = 1,$ and 
\begin{align*}
\E[f(X)] &= \frac{(m+n)np - \big(np(1-p) + n^2 p^2\big)}{mn} \\
&= \frac{p(n+m-np-1+p)}{m}.
\end{align*}
Therefore,
\begin{equation*}
Z(n,2,p,dp) = \frac{p(n+m-np-1+p)}{m},
\end{equation*}
when $m \leq np + 1-p.$  If $m = np + dp$ (so this number is already an integer), then this is equivalent to $p \leq 1/(d+1)$.  In this case,
\begin{equation*}
\Pr[S \geq (n+d)p] = \frac{(n+(d+1)p - 1)}{n+d}.
\end{equation*}
Setting $p = 1/(d+1)$ (to maximize the above) and assuming $m = (n+d)p = (n+d)/(d+1) \in \mathbb{Z}$ gives the simple solution of
\begin{align*}
p_0 &= \frac{d}{n+d}, \\
p_m &= \frac{n}{n+d}, \\
p_n &= 0,
\end{align*}
so that 
\begin{equation*}
\Pr[X_1 + \ldots X_n \geq (n+d)p] = \frac{n}{n+d},
\end{equation*}
which is the same bound given by Markov's inequality.
\subsection{$3$-wise}
For $k = 3$, the expressions are a little messier, so we will omit some of the details on the way to the punchline.  In this case, the support of the optimal solution is $\{p_0, p_m, p_{n-1}, p_n\}$, with
\begin{equation}\label{sol2}
\Pr[S \geq m] = \frac{p\big((n-2)(1-p)^2+m(2-p)\big)}{m},
\end{equation}
as long as $m \leq np + 1-2p$.  The optimal polynomial in the dual problem is
\begin{equation*}
g(x) = \frac{1}{n(n-1)m}\big((n^2+2mn-n-m)xy-(2n+m-1)x^2+x^3\big).
\end{equation*}  
Note that $g$ satisfies the conditions, $g(0) = 0, g(m) = g(n-1) = g(n) = 1,$ and it can be checked that $\E[g(X)]$ equals the quantity in (\ref{sol2}).  Therefore,
\begin{equation*}
Z(n,3,p,dp)  = \frac{p\big((n-2)(1-p)^2+m(2-p)\big)}{m},
\end{equation*}
when $m \leq np + 1-2p.$  If $m = (n+d)p \in \mathbb{Z}$, then this is equivalent to $p \leq 1/(d+2)$, and again $Z$ is maximized with $p$ equal that value.  With these choices, the solution simplifies to
\begin{align*}
p_0 &= \frac{(d+1)^2}{(d+2)(n+d)}, \\
p_m &= \frac{(d+1)n(n-1)}{(n+d)(n+nd-d)}, \\
p_{n-1} &= 0, \\
p_n &= \frac{1}{(d+2)(n+nd-d)},
\end{align*}
so that
\begin{equation*}
\Pr[X_1 + \ldots X_n \geq (n+d)p] = 1- \frac{(d+1)^2}{(d+2)(n+d)}.
\end{equation*}

\bigskip

\noindent {\large {\bf Acknowledgement}}

\bigskip
\noindent Thanks to Swastik Kopparty for his help and insight.

\end{document}